\documentclass[12pt,reqno]{amsart}

\usepackage[margin=2.5cm]{geometry} 

\usepackage{graphicx} 
\usepackage[utf8]{inputenc}
\usepackage[english]{babel}
\usepackage[T1]{fontenc}
\usepackage{amssymb}
\usepackage{mathtools}
\usepackage{amsfonts}
\usepackage{tikz}
\usepackage{tikz-cd}
\usepackage{amsmath}
\usepackage{amsthm}
\usepackage{mathrsfs}
\usepackage{changepage}
\graphicspath{{images/}}
\usepackage{fancyhdr}
\usepackage{csquotes}

\usepackage[backend=biber, sortcites=true, maxbibnames=50]{biblatex}
\usepackage{hyperref}
\hypersetup{
	colorlinks=true,
	linkcolor=blue,
	citecolor=blue,
	urlcolor=blue,
	pdftitle={}
}

\DeclareMathOperator{\diam}{diam}

\newtheorem{theorem}{Theorem}  
\newtheorem{cor}[theorem]{Corollary}  

\theoremstyle{plain}
\newtheorem{thm}{Theorem}
\newtheorem{lemma}[thm]{Lemma}
\newtheorem{corollary}[thm]{Corollary}

\theoremstyle{definition}
\newtheorem{definition}[thm]{Definition}

\newtheorem{remark}[thm]{Remark}
\newtheorem{exmp}[thm]{Example}

\numberwithin{equation}{section}
\numberwithin{thm}{section}

\usepackage{setspace}

\theoremstyle{remark} 
\newtheorem*{ack}{Acknowledgements}
\title{Spaces with multiple conical singularities and their stratifications}
\author{Mohammad Alattar and Lewis Tadman}
\address[Alattar]{Department of Mathematical Sciences, Durham University, United Kingdom}
\email{{mohammad.alattar@durham.ac.uk}}
\address[Tadman]{Department of Mathematical Sciences, Durham University, United Kingdom}
\email{lewis.tadman@durham.ac.uk}
\date{\today}
\begin{document}
\begin{abstract}
    In this paper we further develop the theory of MCS spaces. Our main result shows that MCS spaces, as defined by Perelman, are CS sets with respect to their MCS stratification, and that in fact, the intrinsic stratification agrees with the MCS stratification. As a consequence, we improve on Perelman's result, and settle a question by Fujioka. Furthermore, among other results, we also provide an alternative and short proof to a conjecture by Petersen in all even dimensions.
\end{abstract} 
\maketitle

\section{Introduction}
In the 1970's, Siebenmann \cite{siebenmann1972deformation} introduced two classes of stratified spaces, CS sets (see Definition \ref{Definition: CS sets}) and WCS sets. These are, respectively, locally and weakly cone-like spaces that can be stratified by manifolds. The main difference, roughly speaking, between CS and WCS sets is that CS sets are stratified spaces where the local cone charts are witnessed by homeomorphisms that preserve the strata with respect to some stratification on the link of the cone factor. On the other hand, WCS sets are stratified spaces where the local cone charts preserve only the bottom stratum (strictly speaking, the definition of WCS set allows for more general spaces than cones. These spaces are called mock open cones and will not concern us). CS sets are WCS, and examples of CS sets abound. For instance, all topological manifolds,  all locally finite simplicial complexes and polyhedra are examples of CS sets.

In \cite{perelman1991alexandrov}, Perelman introduced a general class of stratified spaces termed MCS (see Definition \ref{MCS space}). These spaces are, by definition, spaces with multiple conical singularities. Perelman showed that these spaces are examples of WCS sets. Then using the structure of finite-dimensional Alexandrov spaces with curvature bounded below, he showed that Alexandrov spaces are MCS. Combining the geometric structure with the WCS structure, Perelman established his stability theorem \cite{perelman1991alexandrov}. Recently, Fujioka improved on Perelman's result, and among other interesting results, showed that Alexandrov spaces \textit{with the extremal set stratification}, as defined by Perelman--Petrunin \cite{Perelman-Petrunin-93}, are CS sets \cite{fujioka2024alexandrov}. Fujioka inquired whether, for Alexandrov spaces, their MCS stratification differs from the intrinsic stratification and even if so, whether the MCS stratification is CS \cite[Question 3.3]{fujioka2024alexandrov}. Note that, as observed in \cite[Example 2.26]{fujioka2024alexandrov}, the extremal set stratification does not coincide with the MCS stratification, as the double suspension of the Poincaré homology sphere readily shows. 

Our main theorem states that all MCS spaces are CS with their MCS stratification, and in fact, the intrinsic stratification agrees with the MCS stratification, thereby improving on Perelman's result, and answering  Fujioka's question for the more general class of MCS spaces. Note that every Alexandrov space is an MCS space, but not every MCS space is an Alexandrov space.

\begin{theorem}
\label{thm:main.theorem}
Every MCS space with the MCS stratification is a CS set. In fact, the MCS stratification coincides with the intrinsic stratification.

\end{theorem}

Note that the constructed stratification of the links in the CS stratification of the MCS space is not claimed to be a CS stratification. However, the links do admit a CS stratification (cf. Remark \ref{links-need-not-be-cs}). In general, the strata of the links will always be resolvable homology manifolds (see Remark \ref{Cor: Links are resolvable homology manifolds}). Our main tool is a quantity $\mathfrak{I}(X,x)$ which turns out to be the defining quantity for the MCS strata of an MCS space. It is analogous to Armstrong and Handel's notion of an intrinsic dimension $I(X,x)$ \cite{armstrong,handel}. However, in general, for a locally cone-like, finite dimensional space $X$, $\mathfrak{I}(X,x)$ need not coincide with $I(X,x)$ (see Example \ref{invariants-don't-agree}).

Note that as a consequence, since Alexandrov spaces are MCS, we recover \cite[Corollary 3.2]{fujioka2024alexandrov}. Namely, every Alexandrov space with the intrinsic stratification is CS.

We would like to point out the following. If one wishes to prove that the MCS stratification coincides with the intrinsic stratification, then one might be tempted to prove an implication of the form: If $\mathbb{R}^m\times cL$ is MCS  ($L$ is a finite dimensional compactum) then $L$ is MCS. However, Lemma \ref{pathology} shows that this implication is generally untrue. In fact, the situation is worse: $L$ may not even admit a CS stratification. Example \ref{GH-limits} shows that such spaces can arise as Gromov--Hausdorff limits of compact manifolds with a uniform contractibility function. One of the main ideas of the proof of the preceding Theorem also establishes that the link associated to points living in the intrinsic $i$-th stratum of an MCS space can be replaced by a link that is an MCS space (see Link Replacement Lemma, Lemma \ref{replacement lemma}).

The theorem above doesn't extend to spaces for which every point admits a conical neighbourhood (where the link is compact). More precisely, for $m\geq 5$, let $M^m$ be a $m$-manifold with a tame end $\epsilon$. Then, every point in $(M\cup \epsilon)\times \mathbb{R}$ has a neighbourhood (pointed) homeomorphic to the cone over a compactum. The manifold $M$ can be chosen so that $(M\cup \epsilon)\times \mathbb{R}$ does not admit a CS stratification, as shown by Handel \cite[Proposition 3.2]{handel}. This space cannot be an MCS space.

The next corollary shows that Alexandrov spaces have a special tangent microbundle structure along their MCS strata.

\begin{cor}
Let $X$ be a MCS space, $C$ be any component of any MCS stratum, and $y\in C$, and $L$ a link of $y$. Then the canonical diagram $C\rightarrow  C\times X\rightarrow C$ yields a $c(\mathbb{S}^{n-1}\ast L)$ microbundle, where $n$ is the dimension of the stratum. In particular, the map $C\rightarrow C\times X$ is the diagonal map and the map $C\times X\rightarrow C$ is the projection map.
\end{cor}

The proof of the corollary follows immediately from the main theorem in conjunction with \cite[Proposition 3.1]{anderson-hsiang}. In particular, the preceding corollary applies to Alexandrov spaces. Note that the theory of microbundles is an important component of smoothing/PL theory (see for example \cite{lashofimmersion,rudyak}), and more recently, understanding PL structures on Alexandrov spaces has garnered some interest (see for example \cite{fujioka2024alexandrov}). 

Next, we investigate the relationship between Alexandrov spaces that are homology manifolds (see Definition \ref{homology-manifold}), and MCS spaces that are homology manifolds and inquire when MCS spaces appear as Gromov--Hausdorff limits of Riemannian manifolds with certain geometric properties. Understanding when spaces arise as Gromov--Hausdorff limits of Riemannian manifolds under various constraints is a central topic in metric and comparison geometry (see for example \cite{petersen-wilhelm-zhu,cassorla,kapovitch-regularity,alattar-tadman-regularity,DottI,DottII,DottIII,ferryandborus})

In \cite{alattar-tadman-regularity}, the authors observed that an example coming from decomposition space theory illustrates that there there exists a 3-dimensional homology manifold with precisely one non-manifold point that does not admit an Alexandrov metric. This space cannot be an MCS space. The next result records that there exists plenty of examples of odd dimensional MCS spaces that are homology manifolds that are not manifolds and therefore, such spaces cannot be homeomorphic to Alexandrov spaces. 

\begin{theorem}
\label{odd-dimensional-MCS spaces}
Every odd-dimensional Alexandrov space that is a homology manifold is a manifold. However, in each odd dimension $n\geq 5$, there exists simply connected compact MCS spaces that are homology manifolds, homotopy equivalent to spheres that are not manifolds (and hence cannot be homeomorphic to Alexandrov spaces). Furthermore, assume $X$ is one of the two:

\begin{enumerate}
    \item $X$ is a $5$-dimensional simply connected compact MCS space that is a homology manifold, and $d$ is any geodesic metric on $X$.
    \item $X^n$ is a compact MCS space that is a homology manifold  homotopy equivalent to $\mathbb{S}^n$ and $d$ is any geodesic metric on $X$.
\end{enumerate}

Then in either case, there exists a sequence of closed Riemannian manifolds $(M_i,g_i)$ with a uniform dimension $n$, uniform lower volume bound $v>0$, uniform contractibility function $\rho$, and uniform upper diameter bound $D$, such that $(M_i,g_i)\rightarrow_{GH}(X,d)$.

\end{theorem}

The previous result removes the lower curvature bound, compactness, and the fundamental group hypothesis in \cite[Corollary 4.3]{Wu} for odd dimensions.  Observe that the uniform lower volume bound actually follows from the uniform contractibility function property, and the diameter bound follows from convergence. However, we wish to emphasize these geometric properties. 

Examples of spaces satisfying $(2)$ that are not manifolds abound. Indeed, if $Z^n$ is a homology sphere (manifold with the same homology groups as $\mathbb{S}^n$) and $\pi_1(Z)\neq 0$, then $X=\mathrm{Susp}(Z^n)$ is an example of a compact MCS space that is a homology manifold homotopy equivalent to a sphere that  is not a manifold. 

Kervaire showed that there exists many homology spheres in all dimensions $n>4$, that are not simply connected \cite{kervaire}. An example of a $5$ dimensional compact simply connected MCS space that is a homology manifold, that is not an Alexandrov space arises from applying \cite[Remark, p. 68]{kervaire} and then taking the suspension of the resulting homology sphere. Observe further that the geodesic distance condition is not at all restrictive. Indeed, every finite dimensional compact, locally connected continuum admits a geodesic metric, as was shown by Bing \cite{bing-metric}. 

We mention that in \cite{petersen-wilhelm-zhu}, the authors, among other things, investigate whether suspensions over certain exotic spheres appear as Gromov--Hausdorff limits of manifolds with controlled curvature bounds.

In \cite[Theorem 1.2]{fujioka2024alexandrov}, Fujioka showed that the space of directions of any point in an Alexandrov space without proper extremal sets is homeomorphic to a sphere.  In low dimensions, Galaz--García and Guijarro \cite{fernando-with-guijarro} recovered the following result that is implicit in \cite{kapovitch-regularity}: In dimensions $n\leq 4$, the space of directions at any point in an Alexandrov manifold is homeomorphic to a sphere \cite[Corollary 2.5]{fernando-with-guijarro} (namely, in these dimensions, all Alexandrov manifolds are \textit{topologically regular}). In \cite{petersen-problem-list}, Petersen conjectured that the space of directions of any point in a \textit{smoothable} Alexandrov space is homeomorphic to a sphere. In \cite{kapovitch-regularity}, Kapovitch settled the conjecture and more. The next result shows several things. First, it yields a short proof to Petersen's conjecture in even dimensions. In fact, we will show that requiring the Alexandrov space to be merely a topological manifold is enough in such dimensions. Second, it shows Alexandrov manifolds are topologically regular in all even dimensions (instead of dimensions $\leq 4$). Lastly, it allows us to obtain a strengthened version of Fujioka's result in even dimensions (the authors are very thankful to Tadashi Fujioka for pointing out the next result). 

\begin{theorem}
\label{even-dim-alexandrov-mflds}

Let $X^n$ be an even dimensional Alexandrov space that is a topological manifold, $n\geq 4$. Then for every $p\in X$, $\Sigma_p$ is homeomorphic to a sphere.
\end{theorem}

Theorem \ref{even-dim-alexandrov-mflds} is a consequence of Theorem \ref{odd-dimensional-MCS spaces}. Indeed, observe that since $X$ is a topological manifold, for each $p\in X$, by \cite[Theorem 1.1]{Wu}, $\Sigma_p$ is simply connected. Also, $\Sigma_p$ is a homology manifold, and has the same homology groups as the sphere \cite[Proposition 3.1]{Wu}. By Theorem \ref{odd-dimensional-MCS spaces}, $\Sigma_p$ is a topological manifold. Hence, by Whitehead's theorem and the resolution of the Poincaré conjecture, $\Sigma_p$ is topologically a sphere. The theorem now follows. To see how to get the even dimensional version of Petersen's conjecture, observe that if $X^n$ is a smoothable Alexandrov space, then by Perelman's stability theorem it is a manifold \cite{perelman1991alexandrov,Vitali}. Hence if $n$ is even, the theorem immediately gives Petersen's conjecture for even dimensions.

In odd dimensions, the result above is false. Indeed, the double suspension of the Poincaré homology 3-sphere is topologically a sphere by the double suspension theorem due to Edwards and Cannon \cite{edwards-suspensions-of-homology-spheres,cannon}. However, the double suspension has points $o$ such that $\Sigma_o=\mathrm{Susp}(P)$, where $P$ is the Poincaré homology sphere.

The next result, which is immediate from Theorem \ref{odd-dimensional-MCS spaces}, is a sphere theorem generalizing the sphere-theorem due to Grove-Shiohama \cite{groveshiohama}, and attempts to remove the ``no non-manifold points'' hypothesis in \cite[Remark 4.10]{fujioka2024alexandrov} (the authors are very thankful to Galaz--García for discussions surrounding the next Corollary). Sphere theorems are an important component of Alexandrov and Riemannian geometry (see for example \cite{petersen-wilhelm-zhu,Grove-Wilhelm,  radiusspheretheorem,fredwilhelm-collapsingtoalmostiriemannianspaces,grove-and-fred-wilhelm,fujioka2024alexandrov}).

\begin{cor}
\label{sphere-theorem}
Assume $X^n$ is an Alexandrov space that is a homology manifold, $n\geq 4$, $curv\geq 1$ and $\mathrm{diam}(X)>\frac{\pi}{2}$.  

\begin{enumerate}
    \item If $n$ is odd then $X$ is homeomorphic to a sphere.
    \item If $n$ is even then $X$ is homeomorphic  $\mathrm{Susp}(M)$, where $M$ is a closed $n-1$ dimensional  homology sphere with curvature bounded below by $1$ (note $M$ is a manifold, but need not be a sphere).
\end{enumerate}

\end{cor}

We mention in passing that if $n=4$, then an Alexandrov space $X^4$ that is a homology manifold with $\mathrm{curv}\geq 1$ and $\mathrm{diam}(X)> \frac{\pi}{2}$ is either a sphere or the suspension over the Poincaré homology 3-sphere. There are no other examples. To see this, observe that either $\pi_1(M)=0$ or not. If $\pi_1(M)$ is $0$, then $X$ must be topologically a sphere by the Poincaré conjecture. If $\pi_1(M)\neq 0$, then since $M$ has curvature bounded below by $1$, $\pi_1(M)$ is finite (see \cite[Theorem 2.10]{HarveySearle}). Now since the only 3-dimensional homology sphere with non-trivial finite fundamental group is the Poincaré homology sphere (we are very thankful to Galaz--García for pointing this out), we get the result. Also observe that if $X^n$ is an Alexandrov space that is a homology manifold, $n\geq 4$ ($n$ is even), $\mathrm{curv}\geq 1$ and $\diam(X)>\frac{\pi}{2}$ then it is a Gromov--Hausdorff limit of compact Riemannian $n$-manifolds, all homeomorphic to spheres, and all with a uniform contractibility function, uniform lower volume bound, and uniform upper diameter bound (again, we wish to emphasize these properties).

The next Corollary, although well known, is relevant for historical reasons.

\begin{cor}
Let $\{X_i^n\}_{i=1}^{\infty}$ be a sequence of closed $n$-dimensional Riemannian manifolds, $n\geq 5$, with a uniform lower sectional curvature bound $k$, upper diameter bound $D$, and uniform lower volume bound $v$. If $X_i^n\rightarrow_{GH}X^n$, and $n$ is odd, then for all large $i$, $X_i$ is homeomorphic to $X$.
\end{cor}

Note that the previous Corollary is well known for all dimensions as was shown by Perelman \cite{perelman1991alexandrov}. However, the reason we have mentioned it is the following: Historically, when Grove--Petersen--Wu \cite{GPW} proved that the class under consideration above admits a uniform contractibility function, they  were only able to show, using controlled topology, that under the hypothesis above, $X_i$ is homeomorphic to $X_j$ for all large $i,j$ and $X$ is a homology manifold. Then Perelman established the stronger result that $X_i$ is homeomorphic to $X$ for all large $i$. Combining our Theorem \ref{odd-dimensional-MCS spaces}, with the fact that at the time of Grove--Petersen--Wu, it was already known that $X$ in the previous Corollary is a homology manifold, we see that the controlled topology techniques can in fact be used, in a manner that is independent  of Perelman's proof of the stability theorem, to prove that in odd dimensions, $X_i$ is homeomorphic to $X$.

Even though by Theorem \ref{odd-dimensional-MCS spaces}, there exist MCS spaces that are homology manifolds, but not Alexandrov spaces, all such spaces share properties much like in the Alexandrov world. In particular, Theorem \ref{thm:main.theorem}, combined with \cite[Theorem 2, Lemma 1]{henderson}, gives streamlined proofs, and allows us to recover several (well known) results originally due to Wu \cite{Wu,wu-generalizing-edwards-thm-cone-like}. Though, in our case, the result (as stated) is slightly more general, since, for example, in \cite{wu-generalizing-edwards-thm-cone-like}, the links of the cones are, by definition, required to be connected. We note that for Alexandrov spaces, the next corollary also follows  from the recent work of Fujioka \cite{fujioka2024alexandrov}.

\begin{cor}
\label{cor:wu}
If $X$ is an MCS space that is a homology manifold, then 

\begin{enumerate}
    \item $X\times \mathbb{R}$ is a manifold.
    \item The set of non-manifold points of $X$ is discrete.
    \item If $X$ is 3 dimensional, then $X$ is a manifold.
\end{enumerate}
\end{cor}

Note that establishing that certain spaces appearing in metric geometry (that are not necessarily Alexandrov) satisfy the properties above is of interest. For example, recently Fujioka-Gu \cite{fujioka-gu} showed that locally BNPC spaces that are homology manifolds satisfy property $(2)$, generalizing earlier work due to Lytchak and Nagano in the $\mathrm{CAT}(0)$ setting \cite[Theorem 1.2]{LytchakNagano2}. Furthermore, 3-dimensional locally BNPC spaces that are homology manifolds satisfy property (3) above. This can be derived using the work of Thurston \cite[Theorem 3.3]{thurston} and also the work of Fujioka-Gu \cite[Theorem 3.9]{fujioka-gu}. More is true. In fact, if $X$ is a locally BNPC space that is a homology manifold, then $X\times \mathbb{R}$ is indeed a manifold and thus $(1)$ results. This follows from the proof of the topological regularity theorem due to Lytchak and Nagano \cite{LytchakNagano2}. Namely, every 1-strained point in a locally BNPC homology manifold $X$ is a manifold point (see also \cite[Theorem 7.1]{fujioka-gu}). Since every point of $X\times \mathbb{R}$ is obviously 1-strained in the line direction, $X\times \mathbb{R}$ must be a manifold (this was pointed out to us by Tadashi Fujioka, whom we thank gratefully). In fact, one can be content with a topological argument in higher dimensions, and in the ENR setting by appealing to the following result due to Bryant and Lacher \cite[Theorem 1.2]{BryantandLacher}: For $m\geq 5$, every generalized $m$-manifold $X$ with a discrete singular set is a codimension 1 manifold factor (note the dimension restriction).

We shall now provide some additional context and mention related work. The structure of MCS spaces has found several important consequences to the theory of Alexandrov spaces. For example, the notion of orientability for Alexandrov spaces was developed in detail by considering their MCS structure (see for example the work by Harvey--Searle \cite{HarveySearle} and also see the work by Mitsuishi \cite{mitsuishi2016orientabilityfundamentalclassesalexandrov}). The Bonnet--Myers theorem for Alexandrov spaces can be shown to follow from the fact that MCS spaces have universal covers as shown by Harvey and Searle \cite[Theorem 2.10]{HarveySearle}. Similarly,  Harvey and Searle showed that the existence of a ramified double cover for an Alexandrov space follows also from the corresponding existence result for \textit{non-branching} MCS spaces \cite[Theorem 3.4]{HarveySearle}. In \cite[Appendix]{yamaguchi-mitsuishi}, Yamaguchi and Mitsuishi obtained refinements of Perelman's fibration theorem by considering the MCS structure of regular fibers of admissible functions. In general, understanding the structure and properties of the strata of a given stratification on an Alexandrov space have yielded fruitful consequences. For example, Harvey \cite[Section 5]{harvey2016equivariant} showed that extremal sets (which stratify an Alexandrov space) satisfy a tameness property. This property allowed for the application of a generalization of Palais' covering homotopy theorem due to Harvey \cite{g-actions-with-close-orbit-spaces}.

Note that stratified spaces, and CS sets in particular, play an important role in topology. For example, Goresky and MacPherson \cite{Goresky-MacPherson} developed the celebrated theory of intersection homology on CS sets, which has several applications. For example, Schütz \cite{dirkI,dirkII} used intersection homology and stratified spaces to study moduli spaces of a closed linkage. In high dimensions, these spaces turn out to be pseudomanifolds as observed by Schütz (see \cite[section 3]{dirkI} for more details). Additionally, Habegger and Saper generalized intersection homology \cite{Habegger-Leslie}, and established several foundational results. We refer the reader to the books by Banagl \cite{Banagl} and Friedman \cite{Friedman-singular-intersection-homology-29} for a modern treatment on the subject. As another example, Steinberger and West \cite{steinberger-west} proved an equivariant version of Siebenmann's CE-mapping theorem (\cite{approximatingcellularmapsbyhomeossiebenmann}) using the theory of CS sets.

Thus, our results may be viewed as forming a new bridge between metric geometry and the theory of stratified spaces.

\begin{ack}
The authors are extremely grateful to Fernando Galaz--Garcia, Martin Kerin, and Dan Disney for all their comments, criticism and questions during the Durham metric geometry reading seminar. The authors are very thankful to Tadashi Fujioka for his comments and encouraging remarks. Moreover, we are  thankful to Dirk Schütz and Washington Mio for their comments and supportive remarks. Finally, we thank Victoria Pelayo Alvaredo for a careful reading of an earlier version of this manuscript, and for providing several constructive comments and feedback.

\end{ack}

\section{Preliminaries and Definitions}

Before we begin, we denote the vertices of cones  by $o$, and recall that the cone over the empty set is simply a point. First, we give the definition of a filtration, and then define both CS sets and MCS spaces, and mention some of their elementary properties that we shall use. Finally, we introduce an invariant for MCS spaces, that captures, in some sense, how singular a point is relative to the MCS structure. This will be our main tool. We provide examples throughout and devote the majority of this section to studying this object. Let us start with the definition of a locally cone-like space.

\begin{definition}
A space $X$ is said to be \emph{locally cone-like} if for every $p\in X$ there exists a neighbourhood $U$ of $p$, and a compactum $K$ such that $(U,p)\approx (cK,o)$.
\end{definition}
The symbol $\approx$ denotes (pointed) homeomorphism. Now we recall the notion of a stratified set.

\begin{definition}
A \textit{stratified set} $X$ is a metrizable space equipped with a filtration $\varnothing =X^{(-1)}\subseteq \cdots \subseteq X^{(n-1)}\subseteq X^{(n)}\subseteq...\subseteq X$ such that each $X^{(i)}$ is closed in $X$, and the components of the strata, $X[i]= X^{(i)}-X^{(i-1)}$ are open in $X[i]$.
\end{definition}

We will mainly be interested in the case where $X^{(n)}=X$ for some $n$.

For example, a compact manifold $M^n$ with boundary can be stratified by setting $M^{(i)}=\varnothing$ for $i< n-1$, $M^{(n-1)}=\partial{M}$ and $M^{(k)}=M$ for all $k\geq n$. If $M$ is a manifold without boundary, then the stratification is given by $M^{(n)}=M$ and $M^{(k)}=\varnothing$ for $k<n$.

\begin{remark}
We shall often abuse language and merely denote a stratified set by $X$.

\end{remark}

\begin{remark}
Note that the strata in a stratified set are not required to be of the same dimension as their index. In general, these two notions do not coincide.
\end{remark}

\begin{exmp}
We mention a few other examples that we shall need.

\begin{enumerate}
    \item (Product Stratification) If $X$ and $Y$ are stratified sets, then $X\times Y$ is a stratified set, by declaring $(X\times Y)^{(i)}=\bigcup_{a+b=i}X^{(a)}\times Y^{(b)}$. 
    \item (Cone Stratification) If $X$ is a compact stratified set, then $cX$ admits a stratification by defining $(cX)^{(-1)}=\varnothing$, $(cX)^{(0)}=vertex$, and $(cX)^{(i)}=c(X^{(i-1)})$ for $i\geq 1$.

\end{enumerate}

\end{exmp}

Before beginning with the notion of a CS set, we recall the notion of a isomorphism between stratified sets \cite[p. 127]{siebenmann1972deformation}.

\begin{definition}
Let $X$ and $Y$ be stratified sets. An isomorphism $h\colon X\rightarrow Y$  is a homeomorphism such that $h(X^{(n)})=Y^{(n)}$ for all $n\geq 0$. We say that two stratified sets $X$ and $Y$ are isomorphic if there exists an isomorphism between them.
\end{definition}

Now we define the notion of a CS set.

\begin{definition}\label{Definition: CS sets}
  A stratified set $X$ is a CS set if

  \begin{enumerate}
      \item The strata $X[i]$ are topological manifolds of dimension $i$ without boundary.
      \item For each $p\in X$, say in $X[i]$, there exists an open neighbourhood $U$ of $p$ in $X[i]$, and a compact stratified set $L$ and an isomorphism $U\times cL\rightarrow X$ onto a neighbourhood of $p$. The space $L$ is called a link of $p$ in $X$.
  \end{enumerate}
\end{definition}

\begin{remark}
Strictly speaking, in the definition of a CS set, one requires the link $L$ to have \textit{finite formal dimension}. Though since we will be dealing with MCS spaces, this will not concern us. 
\end{remark}

\begin{remark}

\label{links-need-not-be-cs}
One should note that the stratification on $L$ may not be a CS stratification itself. In particular, Siebenmann intentionally did not assume that the stratification of the link $L$ in the definition of a CS set is stratified by manifolds (see \cite[p. 128]{siebenmann1972deformation}).

\end{remark}

\begin{definition}
\label{MCS space}
A metrizable space $X$ is \textit{MCS of dimension $n$} if for every point $p\in X$, there exists a neighbourhood $U_p$ of $p$, a compact MCS space $L$ of dimension $n-1$ and a pointed homeomorphism $(U_p,p)\approx (cL,o)$. The unique MCS space of dimension $-1$ is the empty set.
\end{definition}

Recall, the cone over the empty set is a point. A $0$-dimensional compact MCS space is a finite set of points.

\begin{remark}
Note that the link $L$ in the definition above is metrizable. Therefore $cL$ is metrizable by a metric $d$  such that $d(\alpha y,\alpha z)= \alpha d(y,z)$ for all $\alpha \in [0,\infty)$ and $y,z\in cL$ and $d(\alpha y, \beta y)= |\alpha - \beta||y|$ for any $\alpha, \beta \in [0,\infty)$ and $y\in cL$   \cite[Proposition 3]{gauld}.
\end{remark}

The MCS property is closed under natural operations, as the next example illustrates (cf. \cite[p. 6]{perelman1991alexandrov}) 

\begin{exmp}\label{Example: cone-over-mcs-is-mcs and product-of-mcs-is-mcs}
MCS spaces can be constructed via the following means:
\begin{itemize}
    \item[(1)] (Joins) The join of two compact MCS spaces is MCS.
    \item[(2)] (Products) If $X$ and $Y$ are MCS, then $X\times Y$ is MCS.
    \item[(3)] (Cones) The cone over a compact MCS space is MCS. 
    \item[(4)] (Open Subsets) If $X$ is MCS and $U$ is an open subset of $X$, then $U$ is MCS.

\end{itemize}
\end{exmp}

We mention the following elementary property.

\begin{exmp}[MCS is preserved under homeomorphisms]
If $X$ is an MCS space and $f\colon X\rightarrow Y$ is a homeomorphism, then $Y$ is an MCS space too.

\end{exmp}

Now we recall the definition of a homology manifold.

\begin{definition}[Homology Manifold]
\label{homology-manifold}
A metrizable space $X$ is a homology $n$-manifold if $X$ is locally compact, separable, of finite topological dimension, and for each $x\in X$, $H_{*}(X,X-x)\cong H_{*}(\mathbb{R}^n,\mathbb{R}^n-o)$.
\end{definition}

\section{Main Tool}
We now introduce the main tool of this paper and study its properties which we shall use.

\begin{definition}\label{Definition: Index invarient}
Let $X$ be a locally cone-like finite dimensional metric space. Let $x\in X$. Define $\mathfrak{I}(X,x)$ to be the  supremum of all $m\in \mathbb{N}$ such that there exists a neighbourhood $U_x$ of $x$, a compact MCS space $L$, and a pointed homeomorphism  $\theta_x\colon (U_x,x)\rightarrow (\mathbb{R}^m\times cL,o)$. We shall often refer to $\theta_x$ as the \emph{associated homeomorphism}.

\end{definition}
\begin{remark}
If no such $m$ exists, then $\mathfrak{I}(X,x)=-\infty$.

\end{remark}

The quantity $\mathfrak{I}$  serves as an invariant for MCS spaces, which also detects to some extent the MCS singularities of locally cone-like spaces. This quantity is analogous to Armstrong's notion of an intrinsic dimension at a point \cite{armstrong} in the PL category. Handel used a suitable generalization of the intrinsic dimension \cite{handel} to define the notion of an associated (intrinsic) filtration for a locally cone-like space. The salient difference between the other invariants and ours is that in ours (at least a priori), the links are required to be MCS. Furthermore, $\mathfrak{I}$ turns out to be the defining element for the MCS strata. In general, for a locally cone-like finite dimensional metrizable space $X$, $\mathfrak{I}$ does not coincide with Handel's (see Example \ref{invariants-don't-agree}). Since $\mathfrak{I}$ will be our main tool, we will take the opportunity to study it. As a starting point, we prove the following result.

\begin{lemma}
\label{not-manifold-cone}

Let $X$ be a closed $n$-manifold, where $n\geq 1$. The following are equivalent.

\begin{enumerate}
    \item $cX$ is not a manifold.
    \item $\mathfrak{I}(cX,o)=0$. 
    \item $\mathfrak{I}(cX,o)\leq n$.
\end{enumerate}
\end{lemma}

\begin{proof}
We first shall show that $2\implies 1$ by proving the contrapositive statement. Indeed, if $cX$ is a manifold, then by uniqueness of cone-neighborhoods \cite{Kwun-Uniqueness-of-Cone-Neighborhoods} (cf. \cite{mazur-method-of-infinite-repitition}) $cX$ must be homeomorphic to euclidean space. This gives the implication. As for the forward implication, assume for the sake of obtaining a contradiction that $\mathfrak{I}(cX,o)=m\geq 1$. Thus, there exists an open neighborhood $U$ of $o$ in $cX$, a compact MCS set $L$ and a homeomorphism $(U_o,o)\approx (\mathbb{R}^m\times cL,o)$. Since $\mathbb{R}^m\times cL$ is naturally an open cone, it follows from Stalling's invertible cobordism technique (cf. \cite[Proposition 1]{king-henry}) that $(cX,o)$ is homeomorphic to $(\mathbb{R}^m\times cL,o)$. Let $h\colon (cX,o)\rightarrow (\mathbb{R}^m\times cL,o)$ be a homeomorphism. Since $X$ is a manifold, it follows that $cX-o\approx X\times (0,\infty)$ is a manifold. Hence the only non-manifold point of $cX$ is the vertex. Therefore if we fix any $v\in \mathbb{R}^m\backslash o$, $h^{-1}(v,o)$ is a manifold point. Define $\tau \colon \mathbb{R}^m\times cL\rightarrow \mathbb{R}^m\times cL$ by $\tau(x,l)=(x-v,l)$. Then, $\tau(v,o)=o$ . Now consider $\tau \circ h\colon cX\rightarrow \mathbb{R}^m\times cL$. Then, $\tau h(h^{-1}(v,o))=o$. This implies that $h^{-1}(v,o)$ is not a manifold point, contradiction.  Clearly, $(2)\implies (3)$. Thus it suffices to show that $(3)\implies (1)$. Assume $\mathfrak{I}(cX,o)\leq n$ and that $cX$ is a manifold. Once again, this implies that $cX$ is homeomorphic to euclidean space. Since $X$ has dimension $n$, this implies that $cX$ is homeomorphic to $\mathbb{R}^{n+1}$. Therefore, $\mathfrak{I}(cX,o)=n+1$. Contradiction.

\end{proof}

\begin{exmp}\label{cone-over-rp^2}
Let $X= \mathbb{R}P^2$. Consider $cX$, the open cone over $X$. Let $o$ be the vertex of $cX$. Then, $\mathfrak{I}(cX,o)=0$.  However for $z\neq o$, $\mathfrak{I}(cX,z)\geq 1$. Hence $\mathfrak{I}$ is \textit{not} continuous in general. By continuous we mean that if $x_i\rightarrow x$ then $\mathfrak{I}(X,x_i)\rightarrow \mathfrak{I}(X,x)$.

\end{exmp}

The preceding example is useful to keep in mind when going over the proof of the main theorem. In addition to illustrating the failure of continuity of $\mathfrak{I}$, it demonstrates that singular points will have value strictly smaller than non-singular points.

The next lemma, together with its proof will be used often and implicitly. 

\begin{lemma}
\label{adjusting-lemma}
Assume $X$ is an MCS space. Let $p\in X$, and $U_p$ is a neighbourhood of $p$ in $X$, and $L$ is a compact space. If there exists an open embedding $\varphi\colon (U_p,p)\rightarrow (cL,o)$, then there exists an open neighbourhood $V_p$ of $p$ in $U_p$, and a homeomorphism $(V_p,p)\rightarrow (cL,o)$.

\end{lemma}

\begin{proof}
Consider $\varphi(U_p)$. Since $\varphi$ is an open map, $\varphi(U_p)$ is an open subset of $cL$ containing $o$. Thus there exists an $r>0$ such that $c_rL\subseteq \varphi(U_p)$. Here, $c_rL$ denotes the truncated open cone of height $r$. Thus consider the preimage $\varphi^{-1}(c_rL)$. The restriction
\[
\varphi|_{\varphi^{-1}(c_rL)}\colon \varphi^{-1}(c_rL)\rightarrow c_rL 
\]
is a pointed homeomorphism taking $p$ to $o$.  By a suitable parameterization of the radial coordinate, we obtain the conclusion.

\end{proof}

Let us mention a few more properties concerning $\mathfrak{I}$, which, among other things, show the subtleties that arise with $\mathfrak{I}$.

\begin{lemma}
\label{invariant-under-products}

Let $X,Y$ be MCS spaces. Then for all $x\in X$, and $y\in Y$, 

\[\mathfrak{I}(X\times Y,(x,y))\geq \mathfrak{I}(X,x)+ \mathfrak{I}(Y,y).\]

\end{lemma}

\begin{proof}
Put $m= \mathfrak{I}(X,x)$ and $n= \mathfrak{I}(Y,y)$. Since $X$ is MCS, $\mathfrak{I}(X,x)$ is a maximum (and similarly for $\mathfrak{I}(Y,y)$). Hence there exists conical charts $(U_x,x)\rightarrow (\mathbb{R}^m\times cL,o)$ and $(V_y,y)\rightarrow (\mathbb{R}^n\times c(L'),o)$. Taking products and rearranging gives the result.

\end{proof}

\begin{exmp}[Reverse Inequality] The inequality 

\[\mathfrak{I}(X\times Y,(x,y))\leq \mathfrak{I}(X,x)+ \mathfrak{I}(Y,y)\]

need not hold. Indeed, let $X$ be the suspension over the Poincaré homology sphere. That is, $X= \mathrm{Susp}(P)$ and let $Y= \mathbb{R}$. By the double suspension theorem, it follows that $\mathrm{Susp}(P) \times \mathbb{R}$ is a $5$-manifold. Hence for any $z\in X\times Y$, $\mathfrak{I}(X\times Y,z)=5$.  Now, let $N$ be the north pole of $\mathrm{Susp}(P)$, then $\mathfrak{I}(\mathrm{Susp}(P),N)=0$ and $\mathfrak{I}(\mathbb{R},1)=1$.

\end{exmp}

We shall use the following two lemmas repeatedly, and tacitly.

\begin{lemma}\label{Lemma: I preserved under homeo's}
Let $X$ and $Y$ be MCS spaces. If $f\colon X\rightarrow Y$ is a homeomorphism, then for all $x\in X$, $\mathfrak{I}(X,x)=\mathfrak{I}(Y,f(x))$.

\end{lemma}

\begin{proof}
Let $m= \mathfrak{I}(X,x)$. Let $U_x$ be a neighbourhood of $x$, and let $L$ be a compact MCS space, and 
\[
\theta_x\colon (U_x,x)\rightarrow (\mathbb{R}^m\times cL,o)
\]
the associated pointed homeomorphism such that $\theta_x(x)=o$. Now consider
\[
\theta_x\circ f^{-1}\biggr|_{f(U_x)} \colon (f(U_x),f(x))\rightarrow (\mathbb{R}^m\times cL,o)
. \] It follows that $\mathfrak{I}(Y,f(x))\geq m$. Applying the same argument to $\mathfrak{I}(Y,f(x))$ yields the result.

\end{proof}
The previous Lemma implies that $\mathfrak{I}$ is invariant under homeomorphisms. The converse is not true. 

\begin{exmp}
Let $X= \mathbb{S}^2$ and $Y= \mathbb{T}^2$. Then for each $x\in X$ and $y\in Y$, $\mathfrak{I}(X,x)= \mathfrak{I}(Y,y)$. 

\end{exmp}

We also have the following simple lemma.

\begin{lemma}
\label{preserved-under-open-sets}

If $X$ is an MCS space and $U$ is an open subset of $X$, then for every $u\in U$, $\mathfrak{I}(U,u)=\mathfrak{I}(X,u)$.

\end{lemma}

\begin{proof}
This follows from Lemma \ref{adjusting-lemma}.

\end{proof}

Combining the previous results, one deduces the following.

\begin{corollary}
\label{preserved-under-local-homeos}
Let $X$ and $Y$ be MCS spaces. Assume $f\colon X\rightarrow Y$ is a local homeomorphism. If $x\in X$, then $\mathfrak{I}(X,x)= \mathfrak{I}(Y,f(x))$.

\end{corollary}
We now look at how the invariant $\mathfrak{I}$ behaves under cones.

\begin{lemma}
If $X$ is a compact MCS space, then $\mathfrak{I}(cX,(x,t))\geq \mathfrak{I}(X,x)+1$ for all $x\in X$ and $t>0$.

\end{lemma}

\begin{proof}
Observe, since $t>0$, $\mathfrak{I}(cX,(x,t))= \mathfrak{I}(X\times (0,\infty),(x,t))\geq \mathfrak{I}(X,x)+1$. The last inequality holds in view of Lemma \ref{invariant-under-products}.

\end{proof}
\begin{exmp}
If $X$ is a manifold, then $\mathfrak{I}(cX,(x,t))= \mathfrak{I}(X,x)+1$ for all $x\in X$ and $t>0$. 

\end{exmp}

\begin{exmp}
In general, when $X$ is not a manifold, one does not necessarily have equality in the preceding lemma. More precisely, consider $c(\mathrm{Susp}(P))$, the (open) cone over the suspension of the Poincaré homology sphere $P$. Let $N$ be the north pole of $\mathrm{Susp}(P)$. Then, for $t>0$, such that $(N,t)$ is near, but not the vertex of a cone, $\mathfrak{I}(c(\mathrm{Susp}(P)),(N,t))= \mathfrak{I}(\mathrm{Susp}(P)\times (0,\infty),(N,t))$. Since $\mathrm{Susp}(P) \times (0,\infty)$ is a 5-manifold, the latter expression is simply $5$. However, $\mathfrak{I}(\mathrm{Susp}(P),N)=0$.

\end{exmp}

Finally, we conclude this section by mentioning the following result due to Quinn (see \cite{QuinnEndsofMaps,EndsofmapsIII,quinnresolutions,QuinnErratum}).

\begin{thm}
\label{Quinn's theorem}

Let $X$ be a metric space of (always assumed to be finite) dimension at least 4. The following are equivalent.

\begin{enumerate}
    \item $X$ has a resolution.
    \item $X\times \mathbb{R}^k$ is a manifold for some $k$.
    \item $X\times \mathbb{R}^2$ is a manifold.
\end{enumerate}
\end{thm}

For some applications of this theorem to Alexandrov geometry, see \cite{Wu}.
\section{Proofs and Consequences}

The following theorem establishes Theorem \ref{thm:main.theorem}. We will break the proof down into four parts. First, using the invariant introduced in Definition \ref{Definition: Index invarient}, we define the filtration as
$$
X^{(i)}=\{x\in X:\mathfrak{I}(X,x)\leq i \},
$$
and prove that each $X^{(i)}$ is closed. We then show the $i$-th stratum is an $i$-manifold. Next, we define the stratification on the links, and show that the MCS chart is an isomorphism. We now state the main theorem.

\begin{thm}
\label{main--thm}

Let $X$ be a MCS space. Then there  exists a filtration $\varnothing =X^{(-1)}\subseteq X^{(0)}\subseteq ...\subseteq X^{(N)}= X$  by closed subsets, such that the $i$-th stratum $X[i]$ is an $i$-dimensional manifold, and coincides with the $i$-th MCS stratum. Moreover the following holds: For  each $m$ and every $x\in X[m]$, there exists a neighbourhood $N_x$ of $x$ in $X$, a compact MCS space $L$, a stratification $(L^{(n)})_n$ of $L$ and a homeomorphism $\theta_x\colon N_x\rightarrow \mathbb{R}^m \times cL$ such that 
\[
\theta_x(N_x\cap X^{(k)})= \mathbb{R}^m\times c(L^{(k-m-1)})
\]
for every $k$ whenever defined.

Therefore, with the MCS stratification, $X$ is a CS set. Moreover, the MCS stratification coincides with the intrinsic stratification.

\end{thm}

\begin{remark}
Note that the link $L$ is an MCS space. Hence it admits a CS stratification, but it also admits the stratification $(L^{(n)})_n$.

\end{remark}

\begin{proof}
For each $i\geq 0$, put 
$$
X^{(i)}= \{x\in X: \mathfrak{I}(X,x)\leq i \}
$$

\textbf{Claim 1.} \label{claim 1} For each $i$, $X^{(i)}$ is closed in $X$. 

Notice that one of the difficulties is that $\mathfrak{I}$ is not continuous in general (see Example \ref{cone-over-rp^2}).  Hence the claim is not automatic. The following argument shows how to overcome this problem.

Put $m_p=\mathfrak{I}(X,p)$. Let $U_p,L_p, \theta_p$ be the data coming from the invariant. Namely, 
\[
\theta_p \colon (U_p,p)\rightarrow (\mathbb{R}^{m_p}\times cL_p,o)\] is the associated homeomorphism.  We shall show that for all $q\in U_p$,
\begin{equation}\label{Equation: I(X,q) geq I(X,p)}
\mathfrak{I}(X,q)\geq \mathfrak{I}(X,p).
\end{equation}

To that end, since $\theta_p(q)\in \theta_p(U_p)=\mathbb{R}^{m_p}\times cL$, there exist open sets $U_1$ and $U_2$, such that $\theta_p(q)\in U_1\times U_2\subseteq \theta_p(U_p)$. Since the cone over a compact MCS space is MCS, there exists an open set $U_2'$, a compact MCS space $L'$, and a homeomorphism  $\theta' \colon U_2'\rightarrow cL'$, such that $\theta_p(q)\in U_1\times U_2'\subseteq U_1\times U_2$, and such that if we write $\theta_p(q)=(q_1,q_2)$, then  $\theta'(q_2)=o$. Define $\tau \colon U_1\times U_2'\rightarrow \mathbb{R}^{m_p}\times cL'$  by  $\tau(v,z)=(v-q_1,\theta'(z))$.  Consider 
\[
\theta''= \tau \circ \theta_p \colon\theta_p^{-1}(U_1 \times U_2') \rightarrow \mathbb{R}^{m_p}\times cL'.\] Then, $\theta''(q)=o$. Hence, $\mathfrak{I}(X,q)\geq m_p=\mathfrak{I}(X,p)$. This argument shows that $X^{(i)}$ is closed in $X$. In more detail, observe that if $p\notin X^{(i)}$, then $\mathfrak{I}(X,p)>i$. Forthwith from \ref{Equation: I(X,q) geq I(X,p)} one deduces that for $q$ near $p$, $\mathfrak{I}(X,q)>i$. Whence, $X^{(i)}$ is closed in $X$.

\textbf{Claim 2.} For each $i$, the stratum $X[i]= X^{(i)}-X^{(i-1)}$ is an $i$-dimensional manifold without boundary.

Note for a general finite dimensional locally cone-like space, the ``intrinsic $i$-th stratum'' need not be an $i$-dimensional manifold (as Handel's example referred to in the introduction demonstrates). This suggests to prove the claim, one needs to use the MCS structure. We now proceed with the proof.

Let $p\in X[i]$. Then there exists a compact MCS space $L$, and a homeomorphism 
$$
\theta_p\colon (U_p,p)\rightarrow (\mathbb{R}^{i}\times cL,o).
$$
We will show that $\theta_p(U_p\cap X[i])= \mathbb{R}^i\times o$. To that end,  fix $q\in U_p\cap X[i]$. Write $\theta_p(q)=(q_1,q_2)$. If $q_2=o$ then the inclusion is immediate. Thus we will assume $q_2\neq o$, from which we shall obtain a contradiction: Since $L$ is MCS, there exists an open neighbourhood $U_1$ of $q_1$, $U_2'$ of $q_2$, a compact MCS space $L'$, and an  open embedding  $\theta_p^{-1}(U_1\times U_2')\rightarrow \mathbb{R}^{i+1}\times cL'$, that takes $q$ to $o$ (note that the dimension of the euclidean factor has increased by $1$). Since $q$ is assumed to be in $X[i]$, this gives a contradiction. 
All in all, we have shown $\theta_p(U_p\cap X[i])\subseteq \mathbb{R}^i\times o$.

We shall now establish the reverse inclusion. For that,  let $(x,o)\in \mathbb{R}^i\times o$. Then there exists a unique $q\in U_p$ so that $\theta_p(q)=(x,o)$. We must show that $q\in X[i]$.  Indeed, a completely analogous argument to \eqref{Equation: I(X,q) geq I(X,p)} yields $\mathfrak{I}(X,q)\geq \mathfrak{I}(X,p)=i>i-1$. So,  $q\notin X^{(i-1)}$.  Thus it remains to verify that $\mathfrak{I}(X,q)\leq i$. Assume $\mathfrak{I}(X,q)>i$. Define $\tau \colon \mathbb{R}^i\times cL \rightarrow \mathbb{R}^i \times cL$  by $\tau(v,z)=(v-x,z)$. Put $\theta_q= \tau \circ \theta_p$. Then, $\theta_q(q)=\tau(\theta_p(q))= \tau(x,o)=o$. This implies that $\mathfrak{I}(X,q)>i$ is an absurd statement. Hence the reverse containment follows, and with that, the equality $\theta_p(U_p\cap X[i])= \mathbb{R}^i\times o$ follows. Thus, the strata are indeed topological manifolds of the required dimension. 

Now, let $x\in X[m]$, let $h\colon \mathbb{R}^m\times cL\rightarrow X$ be the associated pointed open embedding. In particular, $h(o)=x$.

\textbf{Claim 3.} There exists a stratification $(L^{(n)})_n$ of $L$ such that $h$ is an isomorphism with respect to $(L^{(n)})_n$, where the stratification on $\mathbb{R}^m\times cL$ is the product stratification.

To prove the claim, first we define a stratification $(\mathbb{R}^m\times cL)^{(p)}$ by setting
$$
(\mathbb{R}^m\times cL)^{(p)}=h^{-1}(h(\mathbb{R}^m\times cL)\cap X^{(p)}).
$$
Observe that by Claim 1 and Claim 2, this indeed forms a stratification for $\mathbb{R}^m\times cL$. We will show that this stratification yields a stratification for $L$. Before doing so, we make the preliminary observation that by  Lemma \ref{preserved-under-open-sets}, and Corollary \ref{preserved-under-local-homeos}, $(\mathbb{R}^m \times cL)^{(p)}$ is precisely the set of points $(x,z)\in \mathbb{R}^m\times cL$ such that $\mathfrak{I}(\mathbb{R}^m\times cL,(x,z))\leq p$. Note that the latter expression is finite because $cL$ is MCS and therefore so is $\mathbb{R}^m\times cL$ (see Example \ref{Example: cone-over-mcs-is-mcs and product-of-mcs-is-mcs}).

Armed with the properties above, we are in a situation where we can adapt the arguments in \cite{handel}. For readers not familiar with the techniques in stratified space theory and in particular with the arguments in \cite{handel}, we shall provide a rather detailed proof on how to adapt Handel's arguments.
To that end, put $L[n]$ to be all the points $l\in L$ such that there exists an $r\in \mathbb{R}^m$ and $t>0$ such that $\mathfrak{I}(\mathbb{R}^m\times cL, (r,l,t))=m+n+1.$

Now we shall verify the following properties:

\begin{enumerate}
    \item $L[k]= \varnothing$ for $k<0,$ 
    \item The filters
    $$
    L^{(n)}= \bigcup_{s\leq n}L[s]
    $$
    are closed in $L$,
    \item The map $h$ satisfies $h(\mathbb{R}^m\times c(L^{(k-m-1)}))= h(\mathbb{R}^m\times cL)\cap X^{(k)}$ whenever defined.
\end{enumerate}

It is easy to see that these properties do indeed yield a stratification for $L$. To prove $(1)$, assume for the sake of obtaining a contradiction that $L[k]\neq \varnothing$ for some $k<0$. So there exists $l\in L$, $r\in \mathbb{R}^m$ and $t>0$ such that

\[\mathfrak{I}(\mathbb{R}^m\times cL,(r,l,t))=m+k+1<m+1.\]By Lemma \ref{Lemma: I preserved under homeo's},  for $r,r'\in \mathbb{R}^m$ and $t,t'\in (0,\infty)$,

\[\mathfrak{I}(\mathbb{R}^m\times cL,(r,l,t))= \mathfrak{I}(\mathbb{R}^m\times cL,(r',l,t')).\]

Consider $x_i=(r,l,\frac{1}{i})$. Then, $x_i\rightarrow (r,o)$. Note, $\mathfrak{I}(\mathbb{R}^m\times cL,x_i)=m+k+1$. So, $x_i\in (\mathbb{R}^m\times cL)^{(m+k+1)}.$ Hence, $(r,o)\in (\mathbb{R}^m\times cL)^{(m+k+1)}$. But, $(r,o)$ lives in $(\mathbb{R}^m\times cL)[m]$. Indeed, to see this, observe that $\mathfrak{I}(\mathbb{R}^m\times cL,(r,o))=\mathfrak{I}(X,x)=m$. So we cannot have $m+k+1<m$. So in particular, $m+k+1\geq m$. We shall now rule out the case $m+k+1=m$ and thus obtain the desired contradiction. Indeed, if $m+k+1=m$ then since $\mathbb{R}^m\times o\subseteq (\mathbb{R}^m\times cL)[m]$, it follows by invariance of domain (since both are manifolds of the same dimension) that $\mathbb{R}^m\times o$ is an open subset of $(\mathbb{R}^m\times cL)[m]$. Hence if $(r,o)\in \mathbb{R}^m\times o$, then for all large $i$, since $x_i \in (\mathbb{R}^m\times cL)[m+k+1]= (\mathbb{R}^m\times cL)[m]$, it follows that $x_i \in \mathbb{R}^m\times o$. This yields a contradiction. Hence $L[k]= \varnothing$ for $k<0$.

Now let us prove $(2)$. By $(1)$, for each $n\geq 0$, the union defining $L^{(n)}$ is finite. Let $l_i\in L^{(n)}$ such that $l_i\rightarrow l$ for some $l\in L$. To demonstrate closedness, we must show that $l\in L^{(n)}$. By definition of $L^{(n)}$, for each $i$, there exists $k_i\leq n$ such that $l_i\in L[k_i]$. Hence there exists $r_i\in \mathbb{R}^m$ and $t_i >0$ so that $\mathfrak{I}(\mathbb{R}^m\times cL, (r_i,l_i,t_i))=m+k_i+1$. By $(1)$, $k= \max(k_i)$ is well defined. Note, $\mathfrak{I}(\mathbb{R}^m\times cL,(r_i,l_i,t_i))= \mathfrak{I}(\mathbb{R}^m\times cL,(0,l_i,r_0))\leq m+k+1$, where $r_0>0$ is fixed. Hence $h(0,l_i,r_0)\rightarrow h(0,l,r_0)$. Thus, $h(0,l,r_0)\in X^{(m+k+1)}$. This follows from Claim 1, which asserts that each $X^{(i)}$ is closed in $X$. Now note, $\mathfrak{I}(X,h(0,l,r_0))=\mathfrak{I}(\mathbb{R}^m\times cL,(0,l,r_0))>m$. To summarize, $m<\mathfrak{I}(\mathbb{R}^m\times cL,(0,l,r_0))\leq m+k+1$, where $k\leq n$. Thus, $l\in L[j]$, for some $j\leq n$.

Point $(3)$ is routine and is essentially contained in the arguments above. For the reader's convenience, we supply the details: First observe that if $k=m$ then the proof of Claim 2 yields the equality

\[h(\mathbb{R}^m\times o)= h(\mathbb{R}^m\times cL)\cap X[m].\]

It remains to show that the latter expression is $h(\mathbb{R}^m\times cL)\cap X^{(m)}$. By definition of $X[m]$, one has $h(\mathbb{R}^m\times o)\subseteq h(\mathbb{R}^m\times cL)\cap X^{(m)}$. To establish the reverse containment, observe that if $\tilde{x}\in h(\mathbb{R}^m\times cL)\cap X^{(m)}$, then $\tilde{x}= h(\tilde{r},\tilde{l},\tilde{t})$.  We aim to show that $\mathfrak{I}(X,\tilde{x})=m$.  If $\mathfrak{I}(X,\tilde{x})\neq m$,
 then $\tilde{t}>0$, but then observe that if we denote $n= \mathfrak{I}(X,\tilde{x})$, then as in the proof of $(1)$ above, since $(\mathbb{R}^m\times cL)^{(n)}$ is closed in $\mathbb{R}^m\times cL$, it follows that the vertex must therefore live in $(\mathbb{R}^m\times cL)^{(n)}$. But this is impossible, since by definition of $h$, it must live in $(\mathbb{R}^m\times cL)[m]$. Hence the case $n\neq m$  cannot occur, and the only possibility is therefore $\mathfrak{I}(X,\tilde{x})=m$, which gives the equality we seek.

Thus it is now enough to consider the case $k>m$. Let $\tilde{x}\in h(\mathbb{R}^m\times c(L^{k-m-1}))$. Then, $\tilde{x}= h(\tilde{r},\tilde{l},\tilde{t})$. We may assume $\tilde{t}>0$ in what follows (because $X^{(m)}\subseteq X^{(k)}$). But now since $\tilde{l}\in L^{(k-m-1)}$, there exists an $0\leq s\leq k-m-1$  such that $\tilde{l}\in L[s]$. Note that the inequality $s\geq 0$ follows from $(1)$ above. Arguing as in the proof of $(1)$, one sees

\[\mathfrak{I}(X,\tilde{x})\leq m+s+1\leq k.\] This shows $h(\mathbb{R}^m\times c(L^{k-m-1}))\subseteq h(\mathbb{R}^m\times cL)\cap X^{(k)}.$ 

To show the reverse inclusion, again observe that if $\tilde{x}\in h(\mathbb{R}^m\times cL)\cap X^{(k)}$ and $\tilde{x}=h(\tilde{r},\tilde{l},\tilde{t})$, where $\tilde{t}>0$ then, $\mathfrak{I}(\mathbb{R}^m\times cL,(\tilde{r},\tilde{l},\tilde{t}))=m+j+1$ for some $j\leq k-m-1$ (note, since $\tilde{t}>0$, $\mathfrak{I}(\mathbb{R}^m\times cL, (\tilde{r},\tilde{l},\tilde{t}))>m$).

Thus, $\tilde{l}\in L[j]$. This shows that $X$ is a CS set with the MCS stratification.

\textbf{Claim 4.} The intrinsic stratification coincides with the MCS stratification.

Since $X$ is MCS, for any $x\in X$, consider the \emph{intrinsic dimension at $x$}, $I(X,x)$. We recall $I(X,x)$ is the supremum of all $m$ such that there exists a neighbourhood $U$ of $x$, a compact space $L$ and an open embedding $h\colon \mathbb{R}^m\times cL\rightarrow X$ onto a neighbourhood of $x$, for which $h(o)=x$. Note that in this case, $L$ is not assumed to be MCS. Before proceeding with the proof, we note that in general, for a locally cone-like space $Z$, $\mathfrak{I}(Z,z)$ need not coincide with $I(Z,z)$ (see Example \ref{invariants-don't-agree}).

In any case, since in the definition of $I(X,x)$, $L$ is merely a compactum, we see that $\mathfrak{I}(X,x)\leq \mathrm{I}(X,x)$.  We shall establish equality. Now we shall verify $\mathrm{I}(X,x)\leq \mathfrak{I}(X,x)$. To that end, let $m= \mathfrak{I}(X,x)$ and $i= \mathrm{I}(X,x)$ ($m$ is for MCS, and $i$ is for intrinsic). So, $x$ lives in the $m$-th MCS stratum. Let $\theta \colon (U_x,x)\rightarrow (\mathbb{R}^i\times cL,o)$ be a pointed homeomorphism, where $L$ is a compactum ($L$ is not necessarily MCS). Now to prove the claim, we must show that $\theta^{-1}(\mathbb{R}^i\times o)$ is contained in the $m$-th MCS stratum. Once we show this, it would then follow from our verification that $X$ is a CS set with respect to the filtration  $(X^{(n)})_n$ that $i\leq m$. This would then conclude the proof. Let $z\in \theta^{-1}(\mathbb{R}^i\times o)$. Note that although $L$ is not assumed to be MCS, $U_x$ is and therefore $\mathbb{R}^i\times cL$ is MCS. Hence

\[\mathfrak{I}(X,z)= \mathfrak{I}(U_x,z)= \mathfrak{I}(\mathbb{R}^i\times cL,\theta(z)).\]

But,

\[\mathfrak{I}(\mathbb{R}^i\times cL,\theta(z))= \mathfrak{I}(\mathbb{R}^i\times cL,o)= \mathfrak{I}(U_x,x)= \mathfrak{I}(X,x)=m.\]

All in all, this shows that $z$ lives in $X[m]$. Hence (since $X$ is a CS set with respect to the MCS stratification) $i\leq m$, and by what we have shown above, this implies  $i=m$, concluding the proof.

\end{proof}

\begin{remark}
\label{Cor: Links are resolvable homology manifolds} The strata $L[k]$ are always resolvable homology manifolds, provided the dimension of $L[k]$ is at least 4.  To show that $L[k]$ are resolvable homology manifolds, first observe that $L[k]$ is a manifold factor (a space $X$ is a manifold factor if $X\times \mathbb{R}^k$ is a manifold for some $k\geq 1$). Indeed, (using the notation in the proof of the preceding theorem), for a natural number $k$, $h(\mathbb{R}^m\times (cL)[k+1])=h(\mathbb{R}^m\times cL)\cap X[k+m+1].$ Note that $h$ is a homeomorphism onto its image.  Further, $\mathbb{R}^m\times (cL)[k+1]$ is $\mathbb{R}^m\times [(cL)^{(k+1)}-(cL)^{(k)}]$, which in turn is homeomorphic to $\mathbb{R}^m\times L[k] \times (0,\infty)$. Since the strata of $X$ are manifolds, it therefore follows that $L[k]$ is a manifold factor.  Hence invoking Quinn's theorem (Theorem \ref{Quinn's theorem}) gives us that $L[k]$ is a resolvable homology manifold whenever its dimension is at least 4. 
    
\end{remark}

Before proceeding, the following lemma is in order, for it is instructive. Namely, it illustrates the types of pathologies that can occur. Before stating the lemma, we mention that a space $X$ is said to be \textit{stably MCS} if $X\times \mathbb{R}^k$ is MCS for some $k\geq 1$.

\begin{lemma}\label{pathology}
\begin{enumerate}

    \item A compactum can be stably MCS but not MCS.
    \item There exists a compact resolvable homology manifold $L$ that does not admit a CS stratification but $\mathbb{R}^i\times cL$ is MCS for each $i$ (hence $L$ cannot be an MCS space).
\end{enumerate}

\end{lemma}

\begin{proof}
For both cases, the construction of the compactum is the same. Indeed, let $\alpha$ be the (non-cellular) Fox-Artin arc in $\mathbb{S}^3$ (see \cite[Corollary 2.10]{Moore}). In particular, $\mathbb{S}^3-\alpha$ is not simply connected. Then $\mathbb{S}^3/\alpha \times \mathbb{R}\approx \mathbb{S}^3\times \mathbb{R}$. The space $\mathbb{S}^3/\alpha$ is the desired example. Indeed, the suspension of $\mathbb{S}^3/\alpha$ is $\mathbb{S}^4$ (see \cite[p. 91]{decompositionsofmanifoldsdaverman} or \cite[Corollary 2.6.4]{embeddingsinmanifolds}). Thus $c(\mathbb{S}^3/\alpha)$ is homeomorphic to $\mathbb{R}^4$. Hence, for each $i$, $\mathbb{R}^i \times c(\mathbb{S}^3/\alpha)\approx \mathbb{R}^{i+4}$ is MCS, however, $\mathbb{S}^3/\alpha$ cannot admit a CS structure by the results of Henderson \cite{henderson}. 
\end{proof}

\begin{remark}For another reference regarding spaces appearing in the proof above and their properties, we refer the reader to \cite{embeddingsinmanifolds}. Though note, in that reference, what they call ``Fox-Artin arc'' is an arc $\alpha$ in $\mathbb{S}^3$ such that $\mathbb{S}^3-\alpha$ is simply connected, and so it is not the arc we are using above. In particular, in their example, $\alpha$ is cellular and hence by Bing-Shrinking, their resulting quotient space is homeomorphic to $\mathbb{S}^3$. For a further reference on Bing-Shrinking, we refer the reader to \cite{discembeddingtheorem}.

\end{remark}

\begin{exmp}
\label{GH-limits}

Note that spaces as in Lemma \ref{pathology}, (2), appear naturally as Gromov--Hausdorff limits. For example, since we have a cell-like map $\mathbb{S}^3\rightarrow \mathbb{S}^3/\alpha$, the proof of \cite[Corollary 2.10]{Moore} shows that there exists a sequence of metrics $d_i$ on $\mathbb{S}^3$ admitting a uniform contractibility function, and a metric $d$ on $\mathbb{S}^3/\alpha$, such that $(\mathbb{S}^3,d_i)\rightarrow_{GH}(\mathbb{S}^3/\alpha,d)$.

\end{exmp}

Note that Lemma \ref{pathology} shows that although $\mathbb{R}^i\times cL$ is MCS, $L$ may not be. However, the proof of the next lemma ensures that in our case, $L$ can always be replaced by an MCS space $L'$. More precisely, we have the following Lemma.

\begin{lemma}[Link Replacement Lemma]
    \label{replacement lemma}

Let $X$ be an MCS space, and let $x$ be a point in $X$. Assume $i$ is the maximal dimension for which there exists a neighbourhood $U$ of $x$, and a compactum $K$ and a pointed homeomorphism $\theta\colon (U_x,x)\rightarrow (\mathbb{R}^i\times cK,o)$.

Then there exist a compact MCS space $L$ and a pointed homeomorphism $(U_x,x)\rightarrow (\mathbb{R}^i\times cL,o)$.

\end{lemma}

\begin{remark}
Again, the space $K$ need not be MCS in general, see Lemma \ref{pathology}.
\end{remark}
\begin{proof}
As the proof of the main theorem demonstrates, since $X$ is an MCS space, $\mathfrak{I}(X,x)=i$. Hence there exists a neighbourhood $N_x$ of $x$ and a compact MCS space $L$ and a pointed homeomorphism $\psi \colon (N_x,x)\rightarrow (\mathbb{R}^i\times cL,o)$. Now by uniqueness of cone neighbourhoods \cite{Kwun-Uniqueness-of-Cone-Neighborhoods} (see also Mazur's elegant generalization \cite{mazur-method-of-infinite-repitition}),  there exists a pointed homeomorphism $(U_x,x)\rightarrow (N_x,x)$. Thus the result follows by composing.

\end{proof}

\begin{exmp}
\label{invariants-don't-agree}
Once again, let $M$ be a manifold with a tame end $\epsilon$ so that $X=(M\cup \epsilon)\times \mathbb{R}$ is locally cone-like and is not the underlying space of a CS set. Then there exists an $x\in X$ so that $\mathfrak{I}(X,x)\neq I(X,x)$. Indeed, otherwise $X$ would be an MCS space. 
\end{exmp}

Now we prove Theorem \ref{odd-dimensional-MCS spaces}. First, we record the following result due to Wu \cite{Wu} (see also \cite{wu-generalizing-edwards-thm-cone-like}).

\begin{thm}[\cite{Wu,wu-generalizing-edwards-thm-cone-like}]

\label{wu's-manifold-recognition-theorem}
An $n$-dimensional Alexandrov space $X$ is a topological manifold if and only if the following two properties are satisfied.

\begin{enumerate}
    \item The space $X$ is a homology manifold.
    \item If $n\geq 4$, for each $p\in X$, the space of directions $\Sigma_p$ is simply connected. 
\end{enumerate}
\end{thm}

First we need a lemma.

\begin{lemma}
\label{lemma-space-of-directions-orientable}
If $X$ is an $n$-dimensional Alexandrov space that is a homology manifold, then for each $p\in X$, $\Sigma_p$, the space of directions at $p$, is orientable.
\end{lemma}

\begin{proof}By \cite[Proposition 3.1]{Wu}, $H_{*}(\Sigma_p)\cong H_{*}(\mathbb{S}^{n-1})$. Since $\Sigma_p$ is a compact Alexandrov space, by \cite[Theorem 1.8]{mitsuishi2016orientabilityfundamentalclassesalexandrov}, $\Sigma_p$ is orientable.

\end{proof}

\begin{exmp}
The previous lemma is not true for general Alexandrov spaces. Indeed, consider $X= c(\mathbb{R}P^2)$. At the vertex $o$, $\Sigma_o= \mathbb{R}P^2$ which is not orientable. In this case, $X$ is not a homology manifold.    
\end{exmp}

Armed with the preceding lemma, we now give a proof to Theorem \ref{odd-dimensional-MCS spaces}.

Before proceeding with the proof we offer a few words on how the idea for the second part arose. In general, to construct an Alexandrov homology manifold that is not a manifold, the standard way is to  look at the suspension over the Poincaré homology sphere. Thus our idea is to investigate the different types of homology spheres that appear, then take their suspensions.

\begin{proof}[Proof of Theorem \ref{odd-dimensional-MCS spaces}]

If $n\leq 3$ then every $n$-dimensional Alexandrov space that is a homology manifold is a manifold.  For $n<3$, this is classical and in fact, for these dimensions, every Alexandrov space is a manifold. For $n=3$, this can be derived through several means, either from the results of Wu \cite{Wu,wu-generalizing-edwards-thm-cone-like}, Fujioka \cite{fujioka2024alexandrov} or our Corollary \ref{cor:wu}. So, we may assume that $n\geq 5$ (the case $n=4$ is irrelevant to the statement of our theorem). To that end, let $X^n$ be an odd-dimensional Alexandrov space that is a homology manifold. By Lemma \ref{lemma-space-of-directions-orientable}, $\Sigma_p$ is orientable. Moreover, $\Sigma_p$ is also even-dimensional and has curvature bounded below by $1$. Therefore by Synge's theorem for Alexandrov spaces \cite[Theorem 2.1]{Petrunin-Parallel-Transport}, $\Sigma_p$ is simply connected and by Wu's theorem, Theorem \ref{wu's-manifold-recognition-theorem}, $X$ is a manifold. This concludes the first part of the result. 

For the second part, let $M$ be an $2n\geq4$ dimensional closed homology sphere with $\pi_1(M)$ infinite. Such a homology sphere can be constructed, for example, using \cite[Theorem 1]{kervaire} and \cite[p. 68]{kervaire} applied to the Higman group \cite{higman}. As observed by Perelman \cite[P. 6]{perelman1991alexandrov}, the join of two compact MCS spaces is an MCS space. Since, $\mathrm{Susp}(M)= M\ast \mathbb{S}^0$, it follows that $\mathrm{Susp}(M)$ is an odd dimensional MCS space. Since $\pi_1\neq 0$, $\mathrm{Susp}(M)$ is not a manifold. Furthermore, since $M$ is a homology sphere, $\mathrm{Susp}(M)$ a simply connected homology manifold. Note that by the first part of the result, it cannot be homeomorphic to an Alexandrov space.

Now we prove the last part of the result. First assume $X^5$ is an compact simply connected compact MCS space that is a homology manifold. By Theorem \ref{thm:main.theorem}, $X$ admits a CS stratification. Hence by \cite[p. 142]{siebenmann1972deformation}, $X$ is an ENR (a Euclidean Neighborhood Retract). Therefore by \cite[Theorem 2]{henderson}, $X$ can be decomposed into a manifold and singular points. Since $X$ is a homology manifold with manifold points, it follows by work of Quinn \cite[Theorem 1.1]{obstructiontoresolutionofhomologymanifolds} that $X$ is a resolvable homology manifold. Now the conclusion follows as explained in the discussion following \cite[Corollary 4]{alattar-tadman-regularity} (with the volume bound condition omitted). Since the argument is not too complicated, we recall it for the reader's convenience. Observe that $H_1(X;\mathbb{Z}_2)=0$. Hence, $\Delta(X)=0$ (here $\Delta(X)$ is from \cite[Corollary 4]{alattar-tadman-regularity} and is the Kirby-Siebenmann invariant of a manifold resolving $X$). In dimension 5, every PL-manifold admits a smooth structure (see for example \cite[Theorem 2]{Milnor}). Therefore by the work of Ferry and Okun \cite{ferryandborus}, we see that we have a sequence $(M_i^5,d_{g_i})$ with a uniform contractibility function such that $(M_i^5,d_{g_i})\rightarrow_{GH}(X,d)$.  It remains to establish the lower volume bound condition. This, however follows from, for example \cite{gromov-filling-riemannian-manifolds} (see the main theorem in \cite{littletop-petersen-greene} for a more general result).

Now, let $X^n$ be a compact MCS space that is a homology manifold and is homotopy equivalent to a sphere. Let $d$ be a geodesic metric on $X$. By Theorem \ref{main--thm}, $X$ is a CS set. Since $X$ is a resolvable homology manifold, there exists a manifold $M$ and a cell-like map $M\rightarrow X$. In particular, $M$ is homotopy equivalent to $X$. Therefore, $M$ is homotopy equivalent to $\mathbb{S}^n$. By the Poincaré conjecture, $M$ is homeomorphic to $\mathbb{S}^n$. Hence, $X$ is the cell-like image of $\mathbb{S}^n$. Thus the result follows as above.

\end{proof}
\begin{remark}
The first part of the result relies on Petrunin's work \cite{Petrunin-Parallel-Transport} on generalizing Synge's theorem to Alexandrov spaces (see also \cite{HarveySearle}). For the lower Ricci curvature bound setting, we would like to point out that in \cite{syngestheoremforriccicurvature}, Wilhelm obtained several  interesting generalizations of Synge's theorem to manifolds with lower Ricci curvature, and intermediate Ricci curvature bounds with $1$-systole conditions.
\end{remark}

Now we prove Corollary \ref{sphere-theorem}.
\begin{proof}{}
Choose $p,q\in X$ such that $|p,q|= \mathrm{diam}(X)$. By Perelman's suspension theorem \cite[Theorem 4.5]{perelman1991alexandrov}, $X$ is homeomorphic to $\mathrm{Susp}(\Sigma_p)$, where $\mathrm{Susp}(\Sigma_p)$ denotes the suspension over the space of directions at $p$. Moreover, $\Sigma_p$ is a homology sphere and is a homology manifold \cite[Proposition 3.1]{Wu}.

Case $\mathrm{I}.$ $\mathrm{dim}(X)$ is odd. 

In this case, by Theorem \ref{odd-dimensional-MCS spaces}, $X$ is a manifold. It is well known that if the suspension is a manifold, it must be a sphere.

Case $\mathrm{II}$. $\mathrm{dim}(X)$ is even.

In this case, $\Sigma_p$ is an odd dimensional homology manifold that is a homology sphere. Thus by Theorem \ref{odd-dimensional-MCS spaces}, it is a manifold.

\end{proof}

\printbibliography
\end{document}